\documentclass{amsart}
\usepackage[margin=3cm]{geometry}
\usepackage{amsmath,amsfonts,amssymb,amsthm}
\usepackage{graphics}
\usepackage{epsfig, esint}
\usepackage{graphics,color}
\usepackage{paralist}

\RequirePackage[colorlinks,citecolor=blue,urlcolor=blue]{hyperref}

   \definecolor{labelkey}{gray}{.8}
   \definecolor{refkey}{gray}{.8}
\providecommand{\R}{\mathbb{R}}

\newcommand{\e}{\varepsilon}

\newcommand{\step}[1]{\medskip\noindent\textbf{Step #1. }}

\newcommand{\ignore}[1]{}

\newtheorem{definition}{Definition}

\newtheorem{theorem}{Theorem}

\newtheorem{lemma}{Lemma}

\newtheorem{assumption}{Assumption}

\author[M. Sch\"affner]{Mathias Sch\"affner}
\address{Technische Universit\"at Dortmund, Fakult\"at f\"ur Mathematik\\
 Vogelpothsweg 87,44227 Dortmund, Germany.}
\email{mathias.schaeffner@tu-dortmund.de}

\title[Higher integrability for variational integrals]{Higher integrability for variational integrals with non-standard growth} 

\begin{document}
\maketitle

%\centerline{\LARGE{NOT FOR DISTRIBUTION}}

\begin{abstract}
 We consider autonomous integral functionals of the form
 \begin{equation*}
\mathcal F[u]:=\int_\Omega f(D u)\,dx \quad\mbox{with}\quad u:\Omega\to\R^N,\, N\geq1,
\end{equation*}
where the convex integrand $f$ satisfies controlled $(p,q)$-growth conditions. We establish higher gradient integrability and partial regularity for minimizers of $\mathcal F$ assuming $\frac{q}p<1+\frac2{n-1}$, $n\geq3$. This improves earlier results valid under the more restrictive assumption $\frac{q}p<1+\frac2{n}$. 
\end{abstract}

\section{Introduction}
In this  note, we study regularity properties of local minimizers of integral functionals
\begin{equation}\label{eq:int}
\mathcal F[u]:=\int_\Omega f(D u)\,dx,
\end{equation}
where $\Omega\subset\R^n$, $n\geq3$, is a bounded domain, $u:\Omega\to\R^N$, $N\geq1$ and $f:\R^{N\times n}\to\R$ is a sufficiently smooth integrand satisfying $(p,q)$-growth of the form 
\begin{assumption}\label{ass:1} There exist $0<\nu\leq L<\infty$ such that $f\in C^2(\R^{N\times n})$ satisfies for all $z,\xi\in \R^{N\times n}$
\begin{equation}\label{ass}
\begin{cases}
\nu |z|^p\leq f(z)\leq L(1+|z|^q),&\\
\nu |z|^{p-2}|\xi|^2\leq \langle \partial^2 f(z)\xi,\xi\rangle\leq L(1+|z|^2)^\frac{q-2}2|\xi|^2.
\end{cases}
\end{equation}
\end{assumption}
Regularity properties of local minimizers of \eqref{eq:int} in the case $p=q$ are classical, see, e.g.,\ \cite{Giu}. A systematic regularity theory in the case $p<q$ was initiated by Marcellini in \cite{Mar89,Mar91}, see \cite{Min06} for an overview. In particular, Marcellini~\cite{Mar91} proves (among other things):
\begin{itemize}
\item[(A)] If $N=1$, $2\leq p<q$ and $\frac{q}p<1+\frac{2}{n-2}$ if $n\geq3$, then every local minimizer $u\in W_{\rm loc}^{1,q}(\Omega)$ of \eqref{eq:int} satisfies $u\in W_{\rm loc}^{1,\infty}(\Omega)$.
\end{itemize}
Local boundedness of the gradient implies that the non-standard growth of $f$ and $\partial^2f$ in \eqref{ass:1} becomes irrelevant and higher regularity (depending on the smoothness of $f$) follows by standard arguments, see e.g.\ \cite[Chapter~7]{Mar89}. However, the $W_{\rm loc}^{1,q}(\Omega)$-assumption on $u$ in (A) is problematic: a priori we can only expect that minimizers of \eqref{eq:int} are in the larger space $u\in W_{\rm loc}^{1,p}(\Omega)$.  Hence, a first important step in the regularity theory for integral functionals with $(p,q)$-growth is to improve gradient integrability for minimizers of \eqref{eq:int}. In \cite{ELM99}, Esposito, Leonetti and Mingione showed
\begin{itemize}
\item[(B)] If $2\leq p<q$ and $\frac{q}p<1+\frac{2}{n}$, then every local minimizer $u\in W_{\rm loc}^{1,p}(\Omega,\R^N)$ of \eqref{eq:int} satisfies $u\in W_{\rm loc}^{1,q}(\Omega,\R^N)$.
\end{itemize}
The combination of (A) and (B) yields unconditional Lipschitz-regularity for minimizers of \eqref{eq:int} in the scalar case under assumption $\frac{q}p<1+\frac2n$, see \cite{BM18} for a recent extension which includes in an optimal way a right-hand side. Only very recently, Bella and the author improved in \cite{BS19c} the results (A) and (B) (in the case $N=1$) in the sense that '$n$' in the assumption on the ratio $\frac{q}p$ can be replaced by '$n-1$' for $n\geq3$ (to be precise, \cite{BS19c,Mar89,Mar91} consider the non-degenerate version \eqref{eq:ass2} of \eqref{ass}). The argument in \cite{BS19c} relies on scalar techniques, e.g., Moser-iteration type arguments, and thus cannot be extended to the vectorial case $N>1$. In this paper, we extend the gradient integrability result of \cite{BS19c} to the vectorial case $N>1$. Before we state the results, we recall a standard notion of local minimizer in the context of functionals with $(p,q)$-growth% (however see e.g.\ the recent result \cite{EMM19} where the convexity assumption is only imposed 'at infinity').
\begin{definition}
We call $u\in W_{\rm loc}^{1,1}(\Omega)$ a local minimizer of  $\mathcal F$ given in \eqref{eq:int} iff
\begin{equation*}
 f(Du)\in L^1_{\rm loc}(\Omega)
\end{equation*}
and
\begin{equation*}
 \int_{{\rm supp}\,\varphi}f(D u)\,dx\leq \int_{{\rm supp}\,\varphi}f(D u+D \varphi)\,dx
\end{equation*}
for any $\varphi\in W^{1,1}(\Omega,\R^N)$ satisfying ${\rm supp}\;\varphi\Subset \Omega$.
\end{definition}
The main result of the present paper is
\begin{theorem}\label{T:1}
Let $\Omega\subset \R^n$, $n\geq3$, and suppose Assumption~\ref{ass:1} is satisfied with $2\leq p<q<\infty$ such that 
\begin{equation}\label{eq:assq2}
\frac{q}{p}<1+\frac2{n-1}.
\end{equation}
Let $u\in W_{\rm loc}^{1,1}(\Omega,\R^N)$ be a local minimizer of the functional $\mathcal F$ given in \eqref{eq:int}. Then, $u\in W_{\rm loc}^{1,q}(\Omega,\R^N)$.
\end{theorem}
%
%\begin{remark}
%We do not know if assumption \eqref{eq:assq2} is  optimal in Theorem~\ref{T:1}. It is known that in the scalar case boundedness and $W^{1,q}_{\rm loc}$-regularity of minimizers can fail if $\frac{q}p$ is to large depending on the dimension $n$, see \cite{G87,Mar91}. 
%%In particular, it is proven in \cite{Mar91} that there exist convex integrands $f$ satisfying \eqref{ass} with $1< p<n-1$ and 
%%%
%%\begin{equation}\label{pq:bounded}
%% q>\frac{(n-1)p}{n-1-p},
%%\end{equation}
%%%
%%which admit unbounded minimizers of the functional \eqref{eq:int}. Very recently, Hirsch and the author showed in \cite{HS19} that \eqref{pq:bounded} is sharp, see also \cite{CMM15,FS93}.
%
%\end{remark}

As mentioned above, higher gradient integrability is a first step in the regularity theory for integral functionals with $(p,q)$-growth, see \cite{CKP14,ELM02,ELM04,BF01} for further higher integrability results under $(p,q)$-conditions. Clearly, we cannot expect to improve from $W_{\rm loc}^{1,q}$ to $W_{\rm loc}^{1,\infty}$ for $N>1$, since this even fails in the classic setting $p=q$, see \cite{SY02}. Direct consequences of Theorem~\ref{T:1} are higher differentiability and a further improvement in gradient integrability in the form:
\begin{itemize}
 \item[(i)] (Higher differentiability). In the situation of Theorem~\ref{T:1} it holds $|\nabla u|^\frac{p-2}2 \nabla u\in W^{1,2}_{\rm loc}(\Omega)$, see Theorem~\ref{T:1b}.
 \item[(ii)] (Higher integrability). Sobolev inequality and (i) imply $\nabla u\in L^{\kappa p}_{\rm loc}(\Omega,\R^{N\times n})$ with $\kappa=\frac{n}{n-2}$. Note that $\kappa p>q$ provided $\frac{q}p<1+\frac{2}{n-2}$.
\end{itemize}
A further, on first glance less direct, consequence of Theorem~\ref{T:1} is partial regularity of minimizers of \eqref{eq:int}, see, e.g.,\,  \cite{AF94,BF01,Breit12,PS}, for partial regularity results under $(p,q)$-conditons. For this, we slightly strengthen the assumptions on the integrand and suppose
%With the higher integrability at hand it is possible to Moreover, the $W^{1,q}_{\rm loc}$-regularity can be used to show partial regularity for minimizers of \eqref{eq:int}. Indeed, in \cite{BF01} (see also \cite{BF05,B12} for related results in the non-autonomous setting) it is  For this, we consider slightly more restrictive assumptions

\begin{assumption}\label{ass:22} There exist $0<\nu\leq L<\infty$ such that $f\in C^2(\R^{N\times n})$ satisfies for all $z,\xi\in \R^{N\times n}$
\begin{equation}\label{eq:ass2}
\begin{cases}
\nu |z|^p\leq f(z)\leq L(1+|z|^q),&\\
\nu (1+|z|^2)^\frac{p-2}2|\xi|^2\leq \langle \partial^2 f(z)\xi,\xi\rangle\leq L(1+|z|^2)^\frac{q-2}2|\xi|^2.
\end{cases}
\end{equation}
\end{assumption}
In \cite{BF01}, Bildhauer and Fuchs prove partial regularity under Assumption~\ref{ass:22} with $\frac{q}p<1+\frac2n$ (\cite{BF01} contains also more general conditions including, e.g., the subquadratic case). Here we show
\begin{theorem}\label{T:2}
Let $\Omega\subset \R^n$, $n\geq3$, and suppose Assumption~\ref{ass:22} is satisfied with $2\leq p<q<\infty$ such that  \eqref{eq:assq2}. 
Let $u\in W_{\rm loc}^{1,1}(\Omega,\R^N)$ be a local minimizer of the functional $\mathcal F$ given in \eqref{eq:int}. Then, there exists an open set $\Omega_0\subset\Omega$ with $|\Omega\setminus\Omega_0|=0$ such that $\nabla u\in C^{0,\alpha}(\Omega_0,\mathbb R^{N\times n})$ for each $0<\alpha<1$.
\end{theorem}
%

%\newpage

\smallskip 

We do not know if \eqref{eq:assq2} in Theorem~\ref{T:1} and \ref{T:2}  is  optimal. Classic counterexamples in the scalar case $N=1$, see, e.g., \cite{G87,Mar91}, show that local boundedness of minimizers can fail if $\frac{q}p$ is to large depending on the dimension $n$. In fact, \cite[Theorem 6.1]{Mar91} and the recent boundedness result \cite{HS19} show that $\frac1p-\frac1q\leq \frac1{n-1}$ is the sharp condition ensuring local boundedness in the scalar case $N=1$ (for sharp results under additional structure assumptions, see, e.g., \cite{CMM15,FS93}). %The counterexamples in \cite{G87,Mar91} also show that $W^{1,q}_{\rm loc}$-regularity of minimizers can fail if $\frac{q}p$ is large enough but these conditions do not match \eqref{eq:assq2}.  

For non-autonomous functionals, i.e., $\int_\Omega f(x,Du)\,dx$, rather precise sufficiently \& necessary conditions are established in \cite{ELM04}, where the conditions on $p,q$ and $n$ has to be balanced with the (H\"older)-regularity in space of the integrand. However, if the integrand is sufficiently smooth in space, the regularity theory in the non-autonomous case essentially coincides with the autonomous case, see \cite{Breit12}.

Currently, regularity theory for non-autonomous integrands with non-standard growth, e.g.\ $p(x)$-Laplacian or double phase functionals are a very active field of research, see, e.g.,\  \cite{BCM18,CFK20,CM15,DM19,Filippis19,RR}.

Coming back to autonomous integral functionals: In \cite{CKP14} higher gradient integrability is proven assuming so-called 'natural' growth conditions, i.e., no upper bound assumption on $\partial^2f$, under the relation $\frac{q}p<1+\frac1{n-1}$. Moreover, in two dimensions we cannot improve the previous results on higher differentiability and partial regularity of, e.g., \cite{BF01,ELM99}, see \cite{BF03} for a full regularity result under Assumption~\ref{ass:22} with $n=2$ and $\frac{q}p<2$.%The method of the present manuscript (combined, e.g.,\ with the approach of \cite{ELM02}) can be used to give a different proof of the results in \cite{CKP14} but does not give an improvement. In two dimension, our method does not improve upon previous results, see \cite{BF03} for a full regularity result in this setting.
\bigskip

Let us briefly describe the main idea in the proof of Theorem~\ref{T:1} and from where our improvement compared to earlier results comes from. The main point is to obtain suitable a priori estimates for minimizers that may already be in  $W_{\rm loc}^{1,q}(\Omega,\R^N)$. The claim then follows by a known regularization and approximation procedure, see, e.g., \cite{ELM99}. For minimizers $v\in W_{\rm loc}^{1,q}(\Omega,\R^N)$ a Caccioppoli-type inequality 
\begin{equation}\label{intro:caccio}
 \int \eta^2 |D(|Dv|^\frac{p-2}2Dv)|^2\lesssim \int |\nabla \eta|^2(1+|Dv|^q)
\end{equation}
is valid for all sufficiently smooth cut-off functions $\eta$, see Lemma~\ref{L:caccio:w1q}. Very formally, the Caccioppoli inequality \eqref{intro:caccio} can be combined with Sobolev inequality and a simple interpolation inequality to obtain
\begin{equation*}
 \|Dv\|_{L^{\kappa p}}^p\lesssim \|D(|Dv|^\frac{p-2}2Dv)\|_{L^2}^2\lesssim \|Dv\|_{L^q}^q\lesssim \|Dv\|_{L^{\kappa p}}^{q\theta}\|Dv\|_{L^{p}}^{(1-\theta)q},
\end{equation*}
where $\theta=\frac{\frac1p-\frac1q}{\frac1{p}-\frac1{\kappa p}}\in(0,1)$ and $\kappa=\frac{n}{n-2}$. The $\|Dv\|_{L^{\kappa p}}$-factor on the right-hand side can be absorbed provided we have $\frac{q\theta}p<1$, but this is precisely the 'old' $(p,q)$-condition $\frac{q}p<1+\frac2n$, this type of argument was previously rigorously implemented in, e.g., \cite{BF01,ELM02}. Our improvement comes from choosing a cut-of function $\eta$ in \eqref{intro:caccio} that is optimized with respect to $v$, which  enables us to use Sobolev inequality on $n-1$-dimensional spheres wich gives the desired improvement, see Section~\ref{sec:proofT:1}. This idea has its origin in joint works with Bella \cite{BS19a,BS19b} on linear non-uniformly elliptic equations. 

\smallskip

With Theorem~\ref{T:1} at hand, we can follows the arguments of \cite{BF01} almost verbatim to prove Theorem~\ref{T:2}. In Section~\ref{sec:Leps}, we sketch (following \cite{BF01}) a corresponding $\e$-regularity result from which Theorem~\ref{T:2} follows by standard methods.

%%%%%%%%%%%%%%%%%%%%%%

\section{Preliminary results}

In this section, we gather some known facts. We begin with a well-known higher differentiability result for minimizers of \eqref{eq:int} under the assumption that $u\in W^{1,q}_{\rm loc}(\Omega,\R^N)$:

\begin{lemma}\label{L:caccio:w1q}
Let $\Omega\subset \R^n$, $n\geq2$, and suppose Assumption~\ref{ass:1} is satisfied with $2\leq p<q<\infty$. Let $v\in W_{\rm loc}^{1,q}(\Omega,\R^N)$ be a local minimizer of the functional $\mathcal F$ given in \eqref{eq:int}. Then, $|Dv|^\frac{p-2}2Dv\in W_{\rm loc}^{1,2}(\Omega,\R^{N\times n})$ and there exists $c=c(\frac{L}\nu,n,N,p,q)\in[1,\infty)$ such that for every $Q\in\mathbb R^{N\times n}$
\begin{equation}\label{est:caccio}
 \int_\Omega \eta^2|D(|Dv|^\frac{p-2}2Dv)|^2\,dx \leq c\int_{\Omega}(1+|Dv|^2)^\frac{q-2}2 |D v-Q|^2|\nabla \eta|^2\,dx\quad\mbox{for all $\eta\in C_c^1(\Omega)$.}
\end{equation}
%
%Suppose in addition that Assumption~\ref{ass:2} is satisfied. Then, $h:=(1+|Dv|^2)^\frac{p}4\in W_{\rm loc}^{1,2}(\Omega)$ and there exists $c=c(\frac{L}\nu,n,N,p,q)\in[1,\infty)$ such that 
%%
%\begin{equation}\label{est:caccio:2}
% \int_\Omega \eta^2|\nabla h|^2\,dx \leq c\int_{\Omega}(1+|Dv|^2)^\frac{q-2}2 |D v|^2|\nabla \eta|^2\,dx\quad\mbox{for all $\eta\in C_c^1(\Omega)$.}
%\end{equation}

\end{lemma}

The Lemma~\ref{L:caccio:w1q} is known, see e.g.\ \cite{BF01,ELM99,Mar91}. Since we did not find a precise reference for estimate \eqref{est:caccio}, we included a prove here following essentially the argument of \cite{ELM99}.

\begin{proof}[Proof of Lemma~\ref{L:caccio:w1q}]
Without loss of generality, we suppose $\nu=1$ the general case $\nu>0$ follows by replacing $f$ with $f/\nu$ (and thus $L$ with $L/\nu$). Throughout the proof, we write $\lesssim$ if $\leq$ holds up to a multiplicative constant depending only on $n,N,p$ and $q$. 

Thanks to the assumption $v\in W^{1,q}_{\rm loc}(\Omega,\R^N)$, the minimizer $v$ satisfies the Euler-Largrange equation
\begin{equation}\label{eq:euler}
 \int_\Omega \langle \partial f(D v), D \varphi\rangle\,dx=0\qquad\mbox{for all $\varphi\in W_0^{1,q}(\Omega,\R^N)$}
\end{equation}
(for this we use that the convexity and growth conditions of $f$ imply $|\partial f(z)|\leq c(1+|z|^{q-1})$ for some $c=c(L,n,N,q,)<\infty$). Next, we use the difference quotient method, to differentiate the above equation: For $s\in\{1,\dots,n\}$, we consider the difference quotient operator
\begin{equation*}
 \tau_{s,h}v:=\tfrac1h(v(\cdot+he_s)-v)\qquad\mbox{where $v\in L_{\rm loc}^1(\R^n,\R^N)$}.
\end{equation*}
Fix $\eta\in C_c^1(\Omega)$. Testing \eqref{eq:euler} with $\varphi:=\tau_{s,-h}(\eta^2(\tau_{s,h}(v-\ell_Q)))\in W_0^{1,q}(\Omega)$, where $\ell_Q(x)=Qx$, we obtain
\begin{align*}
 (I):=&\int_\Omega \eta^2 \langle \tau_{s,h}\partial f(D v),\tau_{s,h}D v\rangle \,dx\\
 =&-2\int_\Omega  \eta \langle\tau_{s,h}\partial f(D v), \tau_{s,h}(v-\ell_Q)\otimes \nabla \eta\rangle\,dx=:(II).
\end{align*}
Writing $\tau_{s,h}\partial f(D v)=\frac1h \partial f(D v+th\tau_{s,h}D v)\big|_{t=0}^{t=1}$, the fundamental theorem of calculus yields
\begin{align}\label{est:Lmar1}
 &\int_\Omega\int_0^1 \eta^2\langle \partial^2f(D v+th\tau_{s,h}D v))\tau_{s,h}D v,\tau_{s,h}D v\rangle\,dt\,dx=(I)\notag\\
 =&(II)=-2\int_\Omega\int_0^1  \eta \langle \partial^2f(D v+th\tau_{s,h}D v)\tau_{s,h}D v, (\tau_{s,h}v-Qe_s) \otimes \nabla \eta\rangle\,dt\,dx,
\end{align}
where we use $\tau_{h,s}\ell_Q=Qe_s$. Youngs inequality yields
\begin{equation}\label{est:Lmar2}
 |(II)|\leq \tfrac12(I)+2(III),%\quad\mbox{where}\quad (III):=\int_\Omega\int_0^1G_{\alpha,k}(\tau_{s,h}u)\langle D^2f(\nabla u+th\tau_{s,h}\nabla u)\nabla \eta,\nabla \eta\rangle\,dt\,dx.
\end{equation}
where
\begin{equation*}
(III):=\int_\Omega\int_0^1\langle \partial^2f(D u+th\tau_{s,h}D u) (\tau_{s,h}v-Qe_s)\otimes \nabla \eta,(\tau_{s,h}v-Qe_s)\otimes \nabla \eta\rangle\,dt\,dx.
\end{equation*}
Combining \eqref{est:Lmar1}, \eqref{est:Lmar2} with the assumptions on $\partial^2f$, see \eqref{ass}, with the elementary estimate
$$
|\tau_{s,h}(|Dv|^{\frac{p-2}2}Dv)|^2\lesssim \int_0^1 |D v+th\tau_{s,h}D v|^\frac{p-2}2 |\tau_{s,h}D v|^2\,dt\
$$
for $h>0$ sufficiently small (see e.g. \cite[Lemma 3.4]{ELM99}), we obtain 
\begin{align}\label{est:Lmar3}
&\int_\Omega\eta^2|\tau_{s,h}(|Dv|^{\frac{p-2}2}Dv)|^2\,dx\notag\\
\lesssim&\int_\Omega\int_0^1 \eta^2|D v+th\tau_{s,h}D v|^\frac{p-2}2 |\tau_{s,h}D v|^2\,dt\,dx\leq (I)\notag\\
\leq& 4(III)\leq4L\int_\Omega\int_0^1 (1+|Dv+th\tau_{s,h}D v|^{q-2})|\nabla \eta|^2|\tau_{s,h}v-Qe_s|^2\,dt\,dx.
\end{align}
Estimate \eqref{est:Lmar3}, the fact  $v\in W_{\rm loc}^{1,q}(\Omega)$ and the arbitrariness of $\eta\in C_c^1(\Omega)$ and $s\in\{1,\dots,n\}$ yield $|Dv|^{\frac{p-2}2}Dv\in W_{\rm loc}^{1,2}(\Omega)$. Sending $h$ to zero in \eqref{est:Lmar3}, we obtain
\begin{equation*}
\int_\Omega\eta^2|\partial_s(|Dv|^{\frac{p-2}2}Dv)|^2\,dx\lesssim L\int_\Omega (1+|Dv|^{q-2})|\nabla \eta|^2|\partial_sv-Qe_s|^2\,dx
\end{equation*}
 the desired estimate \eqref{est:caccio} follows by summing over $s$.

\end{proof}

Next, we state a higher differentiability result under the more restrictive Assumption~\ref{ass:22} which will be used in the proof of Theorem~\ref{T:2}.
\begin{lemma}\label{L:caccio:w1q2}
Let $\Omega\subset \R^n$, $n\geq2$, and suppose Assumption~\ref{ass:22} is satisfied with $2\leq p<q<\infty$. Let $v\in W_{\rm loc}^{1,q}(\Omega,\R^N)$ be a local minimizer of the functional $\mathcal F$ given in \eqref{eq:int}. Then, $h:=(1+|Dv|^2)^\frac{p}4\in W_{\rm loc}^{1,2}(\Omega)$ and there exists $c=c(\frac{L}\nu,n,N,p,q)\in[1,\infty)$ such that for every $Q\in\mathbb R^{N\times n}$
\begin{equation}\label{est:caccio:2}
 \int_\Omega \eta^2|\nabla h|^2\,dx \leq c\int_{\Omega}(1+|Dv|^2)^\frac{q-2}2 |D v-Q|^2|\nabla \eta|^2\,dx\quad\mbox{for all $\eta\in C_c^1(\Omega)$.}
\end{equation}
\end{lemma}
A variation of Lemma~\ref{L:caccio:w1q2} can be found in \cite{BF01} and we only sketch the proof.
\begin{proof}[Proof of Lemma~\ref{L:caccio:w1q2}]
With the same argument as in the proof of Lemma~\ref{L:caccio:w1q} but using \eqref{eq:ass2} instead of \eqref{ass}, we obtain $v\in W^{2,2}_{\rm loc}(\Omega,\R^N)$ and the Caccioppoli inequality
\begin{equation}\label{est:caccio:3}
 \int_\Omega \eta^2(1+|Dv|^2)^\frac{p-2}2|D^2v|^2\,dx \leq c\int_{\Omega}(1+|Dv|^2)^\frac{q-2}2 |D v-Q|^2|\nabla \eta|^2\,dx\quad\mbox{for all $\eta\in C_c^1(\Omega)$,}
\end{equation}
where $c=c(\frac{L}\nu,n,N,p,q)<\infty$. Formally, the chain-rule implies
\begin{equation}\label{formal:chainrule}
|\nabla h|^2\leq c(1+|Dv|^2)^\frac{p-2}2|D^2v|^2,
\end{equation}
where $c=c(n,p)<\infty$, and the claimed estimate \eqref{est:caccio:2} follows from \eqref{est:caccio:3} and \eqref{formal:chainrule}. In general, we are not allowed to use the chain rule, but the above reasoning can be made rigorous: Consider a truncated version $h_m$ of $h$, where $h_m:=\Theta_m(|Dv|)$ with 
\begin{equation*}
 \Theta_m(t):=\begin{cases}(1+t^2)^\frac{p}4&\mbox{if $0\leq t\leq m$}\\(1+m^2)^\frac{p}4&\mbox{if $t\geq m$}\end{cases}.
\end{equation*}
For $h_m$ we are allowed to use the chain-rule and \eqref{est:caccio:3} together with \eqref{formal:chainrule} with $h$ replaced by $h_m$ imply  \eqref{est:caccio:2} with $h$ replaced by $h_m$. The claimed estimate follows by taking the limit $m\to \infty$, see \cite[Proposition~3.2]{BF01} for details.
\end{proof}

The following technical lemma is contained in \cite{BS19c} (see also \cite[proof of Lemma 2.1, Step~1]{BS19a}) and plays a key role in the proof of Theorem~\ref{T:1}
\begin{lemma}[\cite{BS19c}, Lemma 3]\label{L:optimcutoff}
Fix $n\geq2$. For given $0<\rho<\sigma<\infty$ and $v\in L^1(B_{\sigma})$, consider
\begin{equation*}
 J(\rho,\sigma,v):=\inf\left\{\int_{B_{\sigma}}|v||\nabla \eta|^2\,dx \;|\;\eta\in C_0^1(B_{\sigma}),\,\eta\geq0,\,\eta=1\mbox{ in $B_\rho$}\right\}.
\end{equation*}
Then for every $\delta\in(0,1]$ 
\begin{equation}\label{1dmin}
 J(\rho,\sigma,v)\leq  (\sigma-\rho)^{-(1+\frac1\delta)} \biggl(\int_{\rho}^\sigma \left(\int_{\partial B_r} |v|\,d\mathcal H^{n-1}\right)^\delta\,dr\biggr)^\frac1\delta.
\end{equation}
\end{lemma}

For convenience of the reader we include a short proof of Lemma~\ref{L:optimcutoff}
\begin{proof}[Proof of Lemma~\ref{L:optimcutoff}]
Estimate \eqref{1dmin} follows directly by minimizing among radial symmetric cut-off functions. Indeed, we obviously have for every $\e\geq0$
\begin{equation*}
 J(\rho,\sigma,v)\leq \inf\left\{\int_{\rho}^\sigma \eta'(r)^2\left(\int_{\partial B_r}|v|+\e\right)\,dr \;|\;\eta\in C^1(\rho,\sigma),\,\eta(\rho)=1,\,\eta(\sigma)=0\right\}=:J_{{\rm 1d},\e}.
\end{equation*}
For $\e>0$, the one-dimensional minimization problem $J_{{\rm 1d},\e}$ can be solved explicitly and we obtain
\begin{equation}\label{1dmin:2}
J_{{\rm 1d},\e}=\biggl(\int_{\rho}^\sigma \biggl(\int_{\partial B_r}|v|\,d\mathcal H^{n-1}+\e\biggr)^{-1}\,dr\biggr)^{-1}.
\end{equation}
To see \eqref{1dmin:2},  we observe that using the assumption $v\in L^1(B_\sigma)$ and a simple approximation argument we can replace $\eta\in C^1(\rho,\sigma)$ with $\eta\in W^{1,\infty}(\rho,\sigma)$ in the definition of $J_{{\rm 1d},\e}$. Let $\widetilde\eta:[\rho,\sigma]\to[0,\infty)$ be given by
$$\widetilde\eta(r):=1-\biggl(\int_\rho^\sigma b(r)^{-1}\,dr\biggr)^{-1}\int_{\rho}^rb(r)^{-1}\,dr,\quad\mbox{where $b(r):=\int_{\partial B_r}|v|+\e$}.$$
Clearly, $\widetilde \eta\in W^{1,\infty}(\rho,\sigma)$ (since $b\geq\e>0$), $\widetilde \eta(\rho)=1$, $\widetilde \eta(\sigma)=0$, and thus
\begin{equation*}
J_{{\rm 1d},\e}\leq\int_{\rho}^\sigma \widetilde\eta'(r)^2b(r)\,dr=\biggl(\int_{\rho}^\sigma b(r)^{-1}\,dr\biggr)^{-1}.
\end{equation*}
The reverse inequality follows by H\"older's inequality. Next, we deduce \eqref{1dmin} from \eqref{1dmin:2}: For every $s>1$, we obtain by H\"older inequality $\sigma-\rho=\int_\rho^\sigma (\frac{b}{b})^\frac{s-1}s\leq\left(\int_\rho^\sigma b^{s-1}\right)^\frac1s\left(\int_\rho^\sigma \frac1{b}\right)^\frac{s-1}{s}$ with $b$ as above, and by \eqref{1dmin:2} that
\begin{equation*}
J_{{\rm 1d},\e}\leq (\sigma-\rho)^{-\frac{s}{s-1}}\biggl(\int_{\rho}^\sigma \left(\int_{\partial B_r}|v|+\e\right)^{s-1}\,dr\biggr)^{\frac{1}{s-1}}.
\end{equation*}
Sending $\e$ to zero, we obtain \eqref{1dmin} with $\delta=s-1>0$.
\end{proof}
%\newpage

%%%%%%%%%%%%%%%%%%%%%%%%%%%
\section{Higher integrability - Proof of Theorem~\ref{T:1}}\label{sec:proofT:1}

In this section, we prove the following  higher integrability and differentiability result which clearly contains Theorem~\ref{T:1}
\begin{theorem}\label{T:1b}
Let $\Omega\subset \R^n$, $n\geq2$, and suppose Assumption~\ref{ass:1} is satisfied with $2\leq p<q<\infty$ such that $\frac{q}p<1+\min\{\frac2{n-1},1\}$. Let $u\in W_{\rm loc}^{1,1}(\Omega,\R^N)$ be a local minimizer of the functional $\mathcal F$ given in \eqref{eq:int}. Then, $u\in W_{\rm loc}^{1,q}(\Omega,\R^N)$ and $|Du|^\frac{p-2}2Du\in W_{\rm loc}^{1,2}(\Omega,\R^{N\times n})$. Moreover, for  
\begin{equation}\label{def:chi}
\chi=\frac{n-1}{n-3}\quad\mbox{if $n\geq4$}\quad\chi\in(\frac1{2-\frac{q}p},\infty)\quad\mbox{if $n=3$ and}\quad \chi:=\infty\quad\mbox{if $n=2$}.
\end{equation}
there exists $c=c(\frac{L}\nu,n,N,p,q,\chi)\in[1,\infty)$ such that for every $B_{R}(x_0)\Subset \Omega$ 
\begin{equation}\label{est:T:1b2}%
\fint_{B_{\frac{R}2}(x_0)}|Du|^q\,dx+R^2\fint_{B_{\frac{R}2}(x_0)}|D(|Du|^\frac{p-2}2Du)|^2\,dx\leq c \biggl(\fint_{B_R(x_0)}1+f(Du)\,dx\biggr)^\frac{\alpha q}{p}
\end{equation} 
%
%\begin{equation}\label{est:T:1b2}% \biggl(\fint_{B_{\frac{R}2}(x_0)}|Du(x)|^{\frac{np}{n-p}}\,dx\biggr)^\frac{n-p}{np}
%\biggl(\fint_{B_\frac{R}2(x_0)}|Du(x)|^{\kappa p}\,\,dx\biggr)^\frac{1}{\kappa p}\leq c \biggl(\fint_{B_R(x_0)}1+f(Du(x))\,dx\biggr)^\frac{\alpha}p
%\end{equation} 
%%
where 
\begin{equation}\label{def:alpha}
 \alpha:=\frac{1-\frac{q}{\chi p}}{2-\frac{q}p-\frac1\chi}.
\end{equation}

\end{theorem}

\begin{proof}[Proof of Theorem~\ref{T:1b}]

Without loss of generality, we suppose $\nu=1$ the general case $\nu>0$ follows by replacing $f$ with $f/\nu$. Throughout the proof, we write $\lesssim$ if $\leq$ holds up to a multiplicative constant depending only on $L,n,N,p$ and $q$.

Following, e.g., \cite{BF01,ELM99,ELM02}, we consider the perturbed integral functionals
\begin{equation}\label{def:Fsigma}
 \mathcal F_\lambda (w):=\int_\Omega f_\lambda(Dw)\,dx,\qquad \mbox{where}\quad f_\lambda(z):=f(z)+\lambda|z|^q\quad\mbox{with $\lambda\in(0,1)$.}
\end{equation}
We then derive suitable a priori higher differentiability and integrability estimates for local minimizers of $\mathcal F_\lambda$ that are independent of $\lambda\in(0,1)$. The claim then follows with help of a by now standard double approximation procedure in spirit of \cite{ELM99}.

\step 1 One-step improvement. 

Let $v\in W^{1,1}_{\rm loc}(\Omega,\R^N)$ be a local minimizer of the functional $\mathcal F_\lambda$ defined in \eqref{def:Fsigma}, $B_1\Subset\Omega$, and let $\chi>1$ be defined in \eqref{def:chi}. We claim that there exists $c=c(L,n,N,p,q,\chi)\in[1,\infty)$ such that for all $\frac12\leq\rho<\sigma\leq1$ and every $\lambda\in(0,1]$ 
\begin{align}\label{est:1step}
&\int_{B_1}1+f_\lambda(Dv)+\int_{B_\rho}|D(|Dv|^\frac{p-2}2Dv)|^2\,\,dx\notag\\
\leq& \frac{c\biggl(\int_{B_1}1+f_\lambda(Dv)\biggr)^{\frac{\chi}{\chi-1}(1-\frac{q}{\chi p})}}{(\sigma-\rho)^{1+\frac{q}p}}\biggl(\int_{B_1}1+f_\lambda(Dv)+\int_{B_{\sigma}}|D(|Dv|^\frac{p-2}2Dv)|^2\,\,dx\biggr)^{\frac{\chi}{\chi-1}(\frac{q}p-1)}
\end{align}
with the understanding $\frac{\infty}{\infty-1}=1$ and
\begin{align}\label{est:1stepb}
 \int_{B_\rho}|D(|Dv|^\frac{p-2}2Dv)|^2\,\,dx \lesssim \frac{1}{(\sigma-\rho)^2}\frac1\lambda\int_{B_{\sigma}}1+ f_\lambda(Dv)\,dx.
\end{align}
The growth conditions of $f_\lambda$ and the minimality of $v$ imply $v\in W^{1,q}_{\rm loc}(\Omega,\R^N)$ and thus by Lemma~\ref{L:caccio:w1q} 
\begin{equation}\label{est:caccio:1}
 \int_\Omega |D(|Dv|^\frac{p-2}2Dv)|^2\eta^2\,\,dx \lesssim \int_{\Omega}(1+|Dv|^2)^\frac{q-2}2 |D v|^2|\nabla \eta|^2\,dx\quad\mbox{for all $\eta\in C_c^1(\Omega)$.}
\end{equation}
Estimate \eqref{est:1stepb} follows directly from \eqref{est:caccio:1} for $\eta\in C_c^1(B_\sigma)$ with $0\leq \eta\leq1$, $\eta\equiv 1$ on $B_\rho$ and $|\nabla\eta|\leq\frac{2}{\sigma-\rho}$, combined with $|z|^q\leq \frac1\lambda f_\lambda(z)$ and $\lambda\in(0,1]$.

\smallskip

Hence, it is left to show \eqref{est:1step}. For this, we use a technical estimate which follows from Lemma~\ref{L:optimcutoff} and H\"olders inequality: For given $0<\rho<\sigma<\infty$ and $w\in L^q(B_{\sigma})$ it holds
%\partial B_r
\begin{equation}\label{est:aq}
 J(\rho,\sigma,|w|^q)\leq\frac{\biggl(\int_{B_{\sigma}\setminus B_\rho} |w|^p\biggr)^{\frac{\chi}{\chi-1}(1-\frac{q}{\chi p})}}{(\sigma-\rho)^{1+\frac{q}p}} \biggl(\int_{\rho}^{\sigma}\|w\|_{L^{\chi p}(\partial B_r)}^p\,dr\biggr)^{\frac{\chi}{\chi-1}(\frac{q}p-1)},
\end{equation}
where $J$ is defined as in Lemma~\ref{L:optimcutoff}. We postpone the derivation of \eqref{est:aq} to the end of this step. 

Combining \eqref{est:caccio:1} with $(1+|Dv|^2)^\frac{q-2}2 |D v|^2\leq (1+|Dv|)^q$ and estimate \eqref{est:aq} with $w=1+|Dv|$, we obtain
\begin{align}\label{est:caccio2}
 &\int_{B_\rho}|D(|Dv|^\frac{p-2}2Dv)|^2\,\,dx\notag\\ 
 \lesssim& \frac{\biggl(\int_{B_{\sigma}\setminus B_\rho}(1+|D v|)^p\,dx\biggr)^{\frac{\chi}{\chi-1}(1-\frac{q}{\chi p})}}{(\sigma-\rho)^{1+\frac{q}p}}\biggl(\int_{\rho}^{\sigma}\|1+|Dv|\|_{L^{\chi p}(\partial B_r)}^p\,dr\biggr)^{\frac{\chi}{\chi-1}(\frac{q}p-1)}.
\end{align}
Next, we use the Sobolev inequality on spheres to estimate the second factor on the right-hand side in \eqref{est:caccio2}: For $n\geq2$ there exists $c=c(n,N,\chi)\in[1,\infty)$ such that for all $r>0$
\begin{equation}\label{est:sobolev}
 \|D v\|_{L^{\chi p}(\partial B_r)}^p\leq c  r^{(n-1)(\frac1\chi-1)}\biggl(\int_{\partial B_r}|D v|^p\,d\mathcal H^{n-1}+r^2\int_{\partial B_r}|D(|Dv|^{\frac{p-2}2}Dv)|^2\,d\mathcal H^{n-1}\biggr).%\qquad\mbox{for all $r>0$}.
\end{equation}
Combining \eqref{est:sobolev} with elementary estimates and assumption $\frac12\leq \rho<\sigma\leq1$, we obtain
\begin{align}\label{est:rhcaccio}
 \int_{\rho}^{\sigma}\|1+|Dv|\|_{L^{\chi p}(\partial B_r)}^p\,dr\lesssim&\int_{\rho}^{\sigma}1+\|Dv\|_{L^{\chi p}(\partial B_r)}^p\,dr\notag\\
 \lesssim&\int_{\rho}^{\sigma}1+\biggl(\int_{\partial B_r}|D v|^p+|D(|Dv|^{\frac{p-2}2}Dv)|^2\,d\mathcal H^{n-1}\biggr)\,dr\notag\\
 \lesssim&\int_{B_{\sigma}\setminus B_\rho}1+|D v|^p+|D(|Dv|^{\frac{p-2}2}Dv)|^2\,dx.
\end{align}

Combining \eqref{est:caccio2} and estimate \eqref{est:rhcaccio}, we obtain
\begin{align*}
 &\int_{B_\rho}|D(|Dv|^\frac{p-2}2Dv)|^2\,\,dx\\
 \leq& \frac{c\biggl(\int_{B_{1}}(1+|D v|)^p\,dx\biggr)^{\frac{\chi}{\chi-1}(1-\frac{q}{\chi p})}}{(\sigma-\rho)^{1+\frac{q}p}}\biggl(\int_{B_{\sigma}}1+|D v|^p+|D(|Dv|^{\frac{p-2}2}Dv)|^2\,dx\biggr)^{\frac{\chi}{\chi-1}(\frac{q}{p}-1)},
\end{align*}
The claimed estimate \eqref{est:1step} now follows since $|z|^p\leq f(z)\leq f_\lambda(z)$, $\frac{\chi}{\chi-1}(1-\frac{q}{\chi p}+\frac{q}{p}-1)=\frac{q}p\geq1$ and $\int_{B_1}1+f_\lambda(Dv)\,dx\geq|B_1|$. 

\smallskip

Finally, we present the computations regarding \eqref{est:aq}: Lemma~\ref{L:optimcutoff} yields 
\begin{align*}
J(\sigma,\rho,|w|^q)\leq \frac{\biggl(\int_{\rho}^{\sigma}\|w\|_{L^q(\partial B_r)}^{q\delta}\,dr\biggr)^\frac1\delta}{(\sigma-\rho)^{1+\frac1\delta}}\qquad\mbox{for every $\delta>0$.}
\end{align*}
 Using two times the H\"older inequality, we estimate 
\begin{align*}
\biggl(\int_{\rho}^{\sigma}\|w\|_{L^q(\partial B_r)}^{q\delta}\,dr\biggr)^\frac1\delta\leq& \biggl(\int_{\rho}^{\sigma}\|w\|_{L^p(\partial B_r)}^{\theta q\delta}\|w\|_{L^{\chi p}(\partial B_r)}^{(1-\theta) q\delta}\,dr\biggr)^\frac1\delta\quad\mbox{where $\frac{\theta}p+\frac{1-\theta}{\chi p}=\frac1q$}\\
\leq&\biggl(\int_{\rho}^{\sigma}\|w\|_{L^p(\partial B_r)}^{\theta q\delta\frac{s}{s-1}}\,dr\biggr)^\frac{s-1}{s\delta}\biggl(\int_{\rho}^{\sigma}\|w\|_{L^{\chi p}(\partial B_r)}^{(1-\theta) q\delta s}\,dr\biggr)^\frac1{\delta s}\quad\mbox{for every $s>1$.}
\end{align*}
Inequality \eqref{est:aq} follows with the admissible choice
\begin{equation*}
\delta=\frac{p}q\quad\mbox{and}\quad  s=\frac1{1-\theta}\qquad\biggl(\mbox{recall }1-\theta=\frac{\frac1p-\frac1q}{\frac1p-\frac1{\chi p}}\quad \mbox{and $p<q$}\biggr)
\end{equation*}
which ensures $\theta q\delta\frac{s}{s-1}=(1-\theta) q\delta s=p$.

\step 2 Iteration.

We claim that there exists $c=c(L,n,N,p,q,\chi)\in[1,\infty)$ such that
\begin{align}\label{est:key}
\int_{B_\frac12}|Dv|^p+|D(|Dv|^\frac{p-2}2Dv)|^2\,\,dx\leq& c\biggl(\int_{B_1}1+f_\lambda(Dv)\,dx\biggr)^\alpha,%{\frac{\chi-\frac{q}p}{2\chi-1-\chi\frac{q}p}}
\end{align}
where $\alpha$ is defined in \eqref{def:alpha}. For $k\in\mathbb N\cup\{0\}$, we set 
\begin{equation*}
\rho_k=\frac34- \frac{1}{4^{1+k}}\quad\mbox{and}\quad J_k:=\int_{B_1}1+f_\lambda(Dv)+\int_{B_{\rho_k}}|D(|Dv|^\frac{p-2}2Dv)|^2\,\,dx.
\end{equation*}
Estimate \eqref{est:1stepb} and the choice of $\rho_k$ imply for $\lambda\in(0,1]$
\begin{equation}\label{est:Jkbounded}
 \sup_{k\in\mathbb N}J_k\leq \int_{B_1}1+f_\lambda(Dv)+\int_{B_{\frac34}}|D(|Dv|^\frac{p-2}2Dv)|^2\,\,dx\lesssim \frac1\lambda \int_{B_1}1+f_\lambda(Dv)\,dx<\infty.
\end{equation}
From \eqref{est:1step} we deduce the existence of $c=c(L,n,N,p,q,\chi)\in[1,\infty)$ such that for every $k\in\mathbb N$
\begin{equation}\label{est:refine:iterate}
 J_{k-1}\leq c 4^{(1+\frac{q}p) k}\biggl(\int_{B_1}1+f_\lambda(Dv)\biggr)^{\frac{\chi}{\chi-1}(1-\frac{q}{\chi p})}J_k^{\frac{\chi}{\chi-1}\frac{q-p}p}.
\end{equation}
Assumption $\frac{q}p<1+\min\{1,\frac{2}{n-1}\}$ and the choice of $\chi$ yield 
\begin{equation*}
{\frac{\chi}{\chi-1}\frac{q-p}p}\stackrel{\eqref{def:chi}}{=}\begin{cases}\frac{q}p-1&\mbox{if $n=2$}\\
{\frac{\chi}{\chi-1}\frac{q-p}p}&\mbox{if $n=3$}\\
\frac{n-1}2(\frac{q}p-1)&\mbox{if $n\geq4$}\end{cases}<1,
\end{equation*}
where we use for $n=3$ that $\chi\stackrel{\eqref{def:chi}}{>}\frac1{2-\frac{q}p}>0$ and
\begin{equation*}
\frac{\chi}{\chi-1}\frac{q-p}p<1\quad\Leftrightarrow\quad \frac{q-p}p<1-\frac1\chi\quad \Leftrightarrow\quad \frac1\chi<2-\frac{q}p.
\end{equation*}
Hence, iterating \eqref{est:refine:iterate} we obtain (using the uniform bound \eqref{est:Jkbounded} on $J_k$ and ${\frac{\chi}{\chi-1}\frac{q-p}p}<1$)
\begin{eqnarray}\label{est:moser:almostfinal}
& &\int_{B_{\frac12}}|Dv|^p+|D(|Dv|^\frac{p-2}2Dv)|^2\,\,dx\leq J_0\lesssim \biggl(\int_{B_1}1+f_\lambda(Dv)\biggr)^{\frac{\chi}{\chi-1}(1-\frac{q}{\chi p})\sum_{k=0}^\infty ({\frac{\chi}{\chi-1}\frac{q-p}p})^k}
\end{eqnarray}
and the claimed estimate \eqref{est:key} follow from
\begin{equation*}
 \alpha = \frac{\chi}{\chi-1}(1-\frac{q}{\chi p})\sum_{k=0}^\infty ({\frac{\chi}{\chi-1}\frac{q-p}p})^k.
\end{equation*}

\step 3 Conclusion.

We assume $B_1\Subset \Omega$ and show that there exists $c=c(L,n,N,p,q,\chi)\in[1,\infty)$
\begin{equation}\label{est:reduceT2b}
 \int_{B_\frac18}|D u|^{q}\,dx\leq c\left(\int_{B_1}1+f(D u)\,dx\right)^{\frac{\alpha q}p},%\qquad\mbox{where $\kappa=\frac{n}{n-2}$}
\end{equation}
where $\alpha$ is given as in \eqref{def:alpha} above. Clearly, standard scaling, translation and covering arguments yield
$$
\fint_{B_{\frac{R}2}(x_0)}|Du|^q\,dx\leq c \biggl(\fint_{B_R(x_0)}1+f(Du)\,dx\biggr)^\frac{\alpha q}{p}
$$
for all $B_R(x_0)\Subset \Omega$ and $c=c(L,n,N,p,q,\chi)\in[1,\infty)$. The claimed estimate \eqref{est:T:1b2} then follows from Lemma~\ref{L:caccio:w1q}.

\smallskip

Following \cite{ELM99}, we introduce in addition to $\lambda\in(0,1)$ a second small parameter $\e>0$ which is related to a suitable regularization of $u$. For $\e\in(0,\e_0)$, where $0<\e_0\leq1$ is such that $B_{1+\e_0}\Subset\Omega$, we set $u_\e:=u*\varphi_\e$ with $\varphi_\e:=\e^{-n}\varphi(\frac{\cdot}\e)$ and $\varphi$ being a non-negative, radially symmetric mollifier, i.e. it satisfies
$$
\varphi\geq0,\quad {\rm supp}\; \varphi\subset B_1,\quad \int_{\R^n}\varphi(x)\,dx=1,\quad \varphi(\cdot)=\widetilde \varphi(|\cdot|)\quad \mbox{for some $\widetilde\varphi\in C^\infty(\R)$}.
$$
Given $\e,\lambda\in(0,\e_0)$, we denote by $v_{\e,\lambda}\in u_\e+W_0^{1,q}(B_1)$ the unique function satisfying
\begin{equation}\label{eq:defvesigma}
\int_{B_1}f_\lambda(D v_{\e,\lambda})\,dx\leq \int_{B_1}f_\lambda(D v)\,dx\qquad\mbox{for all $v\in u_\e+W_0^{1,q}(B_1)$}.
\end{equation}
Combining Sobolev inequality with the assumption $\frac{q}p<1+\frac{2}{n-2}$ and estimate~\eqref{est:key}, we have 
\begin{eqnarray}\label{est:T2b1}
\biggl(\int_{B_\frac18}|D v_{\e,\lambda}|^{q}\,dx\biggr)^\frac{p}{q}&\lesssim&\int_{B_\frac18}|Dv_{\e,\lambda}|^p+|D(|Dv_{\e,\lambda}|^\frac{p-2}2Dv_{\e,\lambda})|^2\,\,dx\notag\\
&\stackrel{\eqref{est:key}}{\lesssim}&\biggl(\int_{B_1}1+f_\lambda(Dv_{\e,\lambda})\,dx\biggr)^\alpha\notag\\%{\frac{\chi-\frac{q}p}{2\chi-1-\chi\frac{q}p}}
&\stackrel{\eqref{def:Fsigma},\eqref{eq:defvesigma}}{\leq}&\left(\int_{B_{1}}1+f(D u_\e)+\lambda|D u_\e|^q\,dx\right)^\alpha\notag\\
&\leq&\left(|B_1|+\int_{B_{1+\e}}f(D u)\,dx+\lambda\int_{B_1}|D u_\e|^q\,dx\right)^\alpha,
\end{eqnarray} 
where we used Jensen's inequality and the convexity of $f$ in the last step. Similarly,
\begin{align}\label{est:T2b2}
\int_{B_1}|D v_{\e,\lambda}|^p\,dx\stackrel{\eqref{ass}}{\leq}& \int_{B_1}f(D v_{\e,\lambda})\,dx\stackrel{\eqref{def:Fsigma}\eqref{eq:defvesigma}}{\leq} \int_{B_1}f(D u_\e)+\lambda|D u_\e|^q\,dx\notag\\
\leq&\int_{B_{1+\e}}f(D u)\,dx+\lambda\int_{B_1}|D u_\e|^q\,dx.
\end{align}
Fix $\e\in(0,\e_0)$. In view of \eqref{est:T2b1} and \eqref{est:T2b2}, we find $w_\e\in u_\e+W_0^{1,p}(B_1)$ such that as $\lambda\to0$, up to subsequence, % $(v_{\e,\sigma_j})_{j\in\mathbb N}$ such that 
\begin{align*}
v_{\e,\lambda}\rightharpoonup w_\e\qquad\mbox{weakly in $W^{1,p}(B_1)$},\\
D v_{\e,\lambda}\rightharpoonup Dw_\e\qquad\mbox{weakly in $L^{q}(B_\frac{1}8)$}.
\end{align*}
Hence, a combination of \eqref{est:T2b1}, \eqref{est:T2b2} with the weak lower-semicontinuity of convex functionals yield
\begin{align}
 \|D w_\e\|_{L^{q}(B_\frac{1}8)}\leq&\liminf_{\lambda\to0}\|D v_{\e,\lambda}\|_{L^{\kappa p}(B_\frac{1}8)}\lesssim\left(\int_{B_{1+\e}}f(D u)\,dx+1\right)^\frac{\alpha}p\label{est:T2b3}\\
\int_{B_1}|D w_\e|^p\,dx\leq&\int_{B_{1}}f(D w_\e)\,dx\leq\int_{B_{1+\e}}f(D u)\,dx.\label{est:T2b4}
\end{align}
Since $w_\e\in u_\e+W_0^{1,q}(B_1)$ and $u_\e\to u$ in $W^{1,p}(B_1)$, we find by \eqref{est:T2b4} a function  $w\in u+W_0^{1,p}(B_1)$ such that, up to subsequence,
\begin{equation*}
D w_{\e}\rightharpoonup D w\quad\mbox{weakly in $L^p(B_1)$}.
\end{equation*}
Appealing to the bounds \eqref{est:T2b3}, \eqref{est:T2b4} and lower semicontinuity, we obtain
\begin{align}
 \|D w\|_{L^{q}(B_\frac{1}8)}\lesssim&\left(\int_{B_{1}}f(D u)\,dx+1\right)^\frac{\alpha}p\label{est:T2b5}\\
\int_{B_{1}}f(D w)\,dx\leq&\int_{B_{1}}f(D u)\,dx.\label{est:T2b6}
\end{align}
Inequality \eqref{est:T2b6}, strict convexity of $f$ and the fact $w\in u+W_0^{1,p}(B_1)$ imply $w=u$ and thus the claimed estimate \eqref{est:reduceT2b} is a consequence of \eqref{est:T2b5}.

\end{proof}

%%%%%%%%%%%%%

\section{Partial regularity - Proof of Theorem~\ref{T:2}}\label{sec:Leps}

Theorem~\ref{T:2} follows from, the higher integrability statement Theorem~\ref{T:1}, the $\varepsilon$-regularity statement of Lemma~\ref{L:epsreg} below and a well-known iteration argument. 
\begin{lemma}\label{L:epsreg}
Let $\Omega\subset \R^n$, $n\geq3$, and suppose Assumption~\ref{ass:22} is satisfied with $2\leq p<q<\infty$ such that $\frac{q}p<1+\frac2{n-1}$.  Fix $M>0$. There exists $C^*=C^*(n,N,p,q,\frac{L}\nu,M)\in[1,\infty)$ such that for every $\tau\in(0,\frac14)$ there exists $\e=\e(M,\tau)>0$ such that the following is true: Let $u\in W_{\rm loc}^{1,1}(\Omega,\R^N)$ be a local minimizer of the functional $\mathcal F$ given in \eqref{eq:int}. Suppose for some ball $B_r(x)\Subset \Omega$
\begin{equation*}
|(Du)_{x,r}|\leq M,
\end{equation*}
where we use the shorthand $(w)_{x,r}:=\fint_{B_r(x)}w\,dy$, and
\begin{equation*}
 \quad E(x,r):=\fint_{B_r(x)}|Du-(Du)_{x,r}|^2\,dy+\fint_{B_r(x)}|Du-(Du)_{x,r}|^q\,dy\leq \e,
\end{equation*}
then
\begin{equation*}
 E(x,\tau r)\leq C^*\tau^2E(x,r).
\end{equation*}

\end{lemma}
With the higher integrability of Theorem~\ref{T:1b} and the Caccioppoli inequality of Lemma~\ref{L:caccio:w1q2} at hand, we can prove Lemma~\ref{L:epsreg} following almost verbatim the proof of the corresponding result \cite[Lemma 4.1]{BF01}, which contain the statement of Lemma~\ref{L:epsreg} under the assumption $\frac{q}p<1+\frac2{n}$ (note that in \cite{BF01} somewhat more general growth conditions including also the case $1<p<q$ are considered). Thus, we only sketch the argument.

\begin{proof}[Proof of Lemma~\ref{L:epsreg}]

Fix $M>0$. Suppose that Lemma~\ref{L:epsreg} is wrong. Then there exists $\tau\in(0,\frac14)$, a local minimizer $u\in W^{1,1}_{\rm loc}(\Omega,\R^N)$, which in view of Theorem~\ref{T:1} satisfies $u\in W^{1,q}_{\rm loc}(\Omega,\R^N)$, and a  sequence of balls $B_{r_m}(x_m)\Subset B_R$ satisfying
\begin{eqnarray}
 & &|(Du)_{x_m,r_m}|\leq M,\quad E(x_m,r_m)=:\lambda_m\quad\mbox{with}\quad \lim_{m\to\infty}\lambda_m=0,\label{Leps:ass1}\\
 & & E(x_m,\tau r_m)>C^*\tau^2\lambda_m^2,\label{Leps:ass2}
\end{eqnarray}
where $C^*$ is chosen below. We consider the sequence of rescaled functions given by
\begin{align*}
 &v_m(z):=\frac1{\lambda_m r_m}(u(x_m+r_mz)-a_m-r_mA_mz),%\quad\mbox{where}\\
% &a_m:=(u)_{x_m,r_m},\quad A_m:=(Du)_{x_m,r_m}
\end{align*}
where $a_m:=(u)_{x_m,r_m}$ and $A_m:=(Du)_{x_m,r_m}$. Assumption \eqref{Leps:ass1} implies $\sup_m|A_m|\leq M$ and thus, up to subsequence,
\begin{equation*}
A_m\to A\in\R^{N\times n}\label{Leps:limit0}.
\end{equation*}
The definition of $v_m$ yields
\begin{equation}\label{Leps:ass0a}
 Dv_m(z)=\lambda_m^{-1}(Du(x_m+r_mz)-A_m),\quad (v_m)_{0,1}=0,\quad (Dv_m)_{0,1}=0
\end{equation}
Assumptions \eqref{Leps:ass1} and \eqref{Leps:ass2} imply
\begin{align}
 &\fint_{B_1}|Dv_m|^2\,dz+\lambda_m^{q-2}\fint_{B_1}|Dv_m|^q\,dz=\lambda_m^{-1}E(x_m,r_m)=1\label{Leps:ass1a},\\
 &\fint_{B_\tau}|Dv_m-(Dv_m)_{0,\tau}|^2\,dz+\lambda_m^{q-2}\fint_{B_\tau}|Dv_m-(Dv_m)_{0,\tau}|^q\,dz=\lambda_m^{-1}E(x_m,\tau r_m)>C_*\tau^2.\label{Leps:ass1b}
\end{align}
The bound \eqref{Leps:ass1a} together with \eqref{Leps:ass0a} imply the existence of $v\in W^{1,2}(B_1,\R^N)$ such that, up to extracting a further subsequence,
\begin{eqnarray*}
 & v_m\rightharpoonup v&\mbox{in $W^{1,2}(B_1,\R^{ N})$,}\\%\label{Leps:limit1}\\
 &\lambda_m Dv_m\to0&\mbox{in $L^2(B_1,\R^{N\times n})$ and almost everywhere}\\%\label{Leps:limit2}\\
 &\lambda_m^{1-\frac2q}v_m\rightharpoonup 0\quad&\mbox{in $W^{1,q}(B_1,\R^N)$. }%\label{Leps:limit3}
\end{eqnarray*}
The function $v$ satisfies the linear equation with constant coefficients
\begin{equation*}
 \int_{B_1}\langle\partial^2 f(A) Dv,D\varphi\rangle\,dz=0\qquad\mbox{for all $\varphi\in C^1_0(B_1)$,}
\end{equation*}
see, e.g., \cite{Evans86} or \cite[Proposition 4.2]{BF01}. Standard estimates for linear elliptic systems with constant coefficients imply $v\in C_{\rm loc}^\infty(B_1,\R^N)$ and existence of $C^{**}<\infty$ depending only on $n,N$ and the ellipticity contrast of $\partial^2 f(A)$ (and thus on $\frac{L}\nu,p,q,$ and $M$) such that
\begin{equation}\label{epsreg:s2:est}
 \fint_{B_\tau}|Dv-(Dv)_{0,\tau}|^2\leq C^{**}\tau^2.%\quad\mbox{and}\quad\|Dv\|_{L^\infty(B_\rho)}^2\leq \frac{c}{(1-\rho)^2}\qquad\mbox{for all $\rho\in(0,1)$}.
\end{equation}
Choosing $C^*=2C^{**}$ we obtain a contradiction between \eqref{Leps:ass1b} and \eqref{epsreg:s2:est} provided we have as $m\to\infty$
\begin{eqnarray}
& Dv_m\to Dv\quad&\mbox{in $L^2_{\rm loc}(B_1)$},\label{claim:strong:1}\\
& \lambda_m^{1-\frac2q}Dv_m\to0\quad&\mbox{in $L^q_{\rm loc}(B_1)$.}\label{claim:strong:2}
\end{eqnarray}
Exanctly as in \cite[Proposition 4.3]{BF01} (with $\mu=2-p$, see also \cite[Section~3.4.3.2]{B03} for a more detailed presentation of the proof), we have for all $\rho\in(0,1)$,
\begin{equation}\label{claim:convergence:wm}
 \lim_{m\to\infty}\int_{B_\rho}\int_0^1(1-s)\biggl(1+|A_m+\lambda_m(Dv+sDw_m)|^2\biggr)^\frac{p-2}2|Dw_m|^2\,dz=0,
\end{equation}
where $w:=v_m-v$, and thus the local $L^2$-convergence \eqref{claim:strong:1} follows. It is left to prove \eqref{claim:strong:2}. For this, we introduce for $\rho\in(0,1)$ and $T>0$ the sequence of subsets
\begin{equation*}%\label{def:Um}
 U_m:=U_m(\rho,T):=\{\,z\in B_\rho\,:\,\lambda_m|Dv_m|\leq T\,\}.
\end{equation*}
The local Lipschitz regularity of $v$, $q>2$ and \eqref{claim:strong:1} imply for all $\rho\in(0,1)$ and $T>0$
\begin{align*}
\limsup_{m\to\infty}\int_{U_m(\rho,T)}\lambda_m^{q-2}|Dv_m|^q\,dz\lesssim& \limsup_{m\to\infty}\int_{U_m(\rho,T)}\lambda_m^{q-2}|Dw_m|^q\,dz\\
\lesssim&\limsup_{m\to\infty}\int_{B_\rho}(M^{q-2}+\lambda_m^{q-2}|Dv|^{q-2})|Dw_m|^2\,dz=0,
\end{align*}
where here and for the rest of the proof $\lesssim$ means $\leq$ up to a multiplicative constant depending only on $L,n,N,p$ and $q$.
Hence, it is left to show that there exists $T>0$ such that
\begin{equation*}
 \limsup_{m\to\infty}\int_{B_\rho\setminus U_m(\rho,T)}\lambda_m^{q-2}|Dv_m|^q\,dz\leq0\quad\mbox{for all $\rho\in(0,1)$}.
\end{equation*}
As in \cite{BF01}, we introduce a sequence of auxiliary functions
\begin{equation*}
\psi_m:=\lambda_m^{-1}\biggl[(1+|A_m+\lambda_m Dv_m|^2)^\frac{p}4-(1+|A_m|^2)^{\frac{p}4}\biggr],
\end{equation*}
which satisfy
\begin{equation}\label{est:psiw12}
\limsup_{m\to\infty}\|\psi_m\|_{W^{1,2}(B_\rho)}\lesssim c(\rho)\in[1,\infty)\qquad\mbox{for all $\rho\in(0,1)$}.
\end{equation}
Indeed, by Theorem~\ref{T:1} and Lemma~\ref{L:caccio:w1q2}, we have for every $\rho\in(0,1)$ and every $Q\in\mathbb R^{N\times n}$
\begin{equation*}
 \int_{B_{\rho r_m}(x_m)}|\nabla (1+|D u(x)|^2)^\frac{p}4|^2\,dx\lesssim r_m^{-2}c(\rho)\int_{B_{r_m}(x_m)}(1+|\nabla u(x)|)^{q-2}|Du(x)-Q|^2\,dx
\end{equation*} 
and thus by rescaling and setting $Q=A_m$
\begin{equation*}
 \int_{B_\rho}|\nabla \psi_m|^2\,dz\lesssim c(\rho)\int_{B_1}(1+|A|^{q-2}+|\lambda_m Dv_m|^{q-2}))|Dv_m|^2\,dz\stackrel{\eqref{Leps:ass1a}}{\lesssim} c(\rho)(1+M^{q-2}).
\end{equation*} 
The identity $\psi_m=\lambda_m^{-1}\int_0^1\frac{d}{dt}\Theta(A_m+t\lambda_mv_m)\,dt$ with $\Theta(F):=(1+|F|^2)^\frac{p}4$ implies
\begin{equation*}
 |\psi_m|\leq c(|Dv_m|+\lambda_m^{\frac{p-2}2}|Dv_m|^{\frac{p}2})
\end{equation*}
(see \cite[p.\ 555]{BF01} for details) and thus with help of \eqref{claim:convergence:wm}, we obtain
\begin{equation*}
 \limsup_{m\to\infty}\int_{B_\rho}|\psi_m|^2\,dz\lesssim c(\rho).
\end{equation*}
For $T$ sufficiently large (depending on $M$) there exists $c>0$ such that for all $z\in B_\rho\setminus U_m(\rho,T)$
\begin{align*}
 \psi_m(z)\geq c\lambda_m^{-1}\lambda_m^{\frac{p}2}|Dv_m(z)|^\frac{p}2\quad\mbox{ and thus }\quad\lambda_m^{2(1+\frac{q}p)}\psi_m^{\frac{2q}p}(z)\geq c^{\frac{2q}p}\lambda_m^{q-2}|Dv_m(z)|^q
\end{align*}
Estimate \eqref{est:psiw12} and Sobolev embedding imply $\limsup_{m\to\infty}\|\psi_m\|_{L^{\frac{2n}{n-2}}(B_\rho)}\lesssim c(\rho)\in[1,\infty)$. Hence, using assumption $\frac{q}p<1+\frac{2}{n-1}$ (and thus $\frac{2q}p<\frac{2n}{n-2}$), we obtain for every $\rho\in(0,1)$
\begin{align*}
\limsup_{m\to\infty}\int_{B_\rho\setminus U_m(\rho,T)}\lambda_m^{q-2}|Dv_m|^q\,dz\lesssim \lambda_m^{2(1+\frac{q}p)}\int_{B_\rho}\psi_m^{\frac{2q}p}(z)\,dz\lesssim c(\rho)\limsup_{m\to\infty}\lambda_m^{2(1+\frac{q}p)}=0,
\end{align*}
which finishes the proof.

\end{proof}
%
%
%\section*{Acknowledgments}
%
%The authors were supported by the German Science Foundation DFG in context of the Emmy Noether Junior Research Group BE 5922/1-1.
%%\appendix

%\section{Proof of Lemma~\ref{L:Mar}}

\end{document}